\documentclass[11pt]{amsart}%
\usepackage{graphicx}
\usepackage{amscd}
\usepackage{amsmath}
\usepackage{amsfonts}
\usepackage{amssymb}%
\providecommand{\U}[1]{\protect\rule{.1in}{.1in}}
\newtheorem{theorem}{Theorem}
\theoremstyle{plain}

\newtheorem{definition}{Definition}[section]

\newtheorem{remark}{Remark}[section]

\numberwithin{equation}{section}
\numberwithin{theorem}{section}

\begin{document}
\title[The Semi-Periodic Case]{Left-Definite Variations of the Classical Fourier Expansion Theorem, Part II }
\author{L. L. Littlejohn}
\address{Department of Mathematics, Baylor University, Waco, Texas, 76706}
\email{lance\_littlejohn@baylor.edu}
\author{E. L. Smith}
\address{Department of Mathematics, Baylor University, Waco, Texas, 76706}
\email{edward\_smith1@baylor.edu}
\author{A. Zettl}
\address{Department of Mathematics, Northern Illinois University, DeKalb, Illinois, 60115-2880}
\email{zettl@msn.com}
\date{September 1, 2021 (Left-Definite Theory and the Semi-Periodic Fourier Case)}
\dedicatory{ }
\begin{abstract}
In 2002, Littlejohn and Wellman developed a general left-definite theory for
arbitrary self-adjoint operators in a Hilbert space that are bounded below by
a positive constant. Zettl and Littlejohn, in 2005, applied this general
theory to the classical second-order Fourier operator with periodic boundary
boundary conditions. In this paper, we construct sequences of left-definite
Hilbert spaces $\{H_{n}\}_{n\in\mathbb{N}}$ and left-definite self-adjoint
operators $\{A_{n}\}_{n\in\mathbb{N}}$ associated with the Fourier operator
with semi-periodic boundary conditions. We obtain explicit formulas for the
domain of the square root of the self-adjoint operator $A$ obtained from this
boundary value problem as well as explicit representations of the domains
$\mathcal{D}(A^{n/2})$ for all positive integers $n.$ Furthermore, a Fourier
expansion theorem is given in each left-definite space $H_{n}.$

\end{abstract}
\maketitle

\section{Introduction}

For a self-adjoint operator $A$ in a Hilbert space $H$, bounded below in $H$
by a positive constant, Littlejohn and Wellman \cite{LW} construct a continuum
of Hilbert spaces $\{H_{r}\}_{r>0}$ and a continuum of self-adjoint operators
$\{A_{r}\}_{r>0}$ from the pair $(H,A)$. For each $r>0,$ $H_{r}$ is called the
$r^{th}$ left-definite Hilbert space associated with $(H,A)$ and $A_{r}$ is
called the $r^{th}$ left-definite operator associated with $(H,A).$ These
spaces and operators share many properties that the original operator $A$ and
the Hilbert space $H$ satisfy. Indeed, the spectrum of each $A_{r}$ coincides
with the spectrum of $A$, eigenfunctions of $A$ are also eigenfunctions of
each $A_{r}$ and, in particular, a complete orthogonal set of eigenfunctions
of $A$ in $H$ is also a complete set of eigenfunctions of each $A_{r}$ in
$H_{r}.$

This general theory has been applied to several classical singular
second-order differential equations, including the Hermite \cite{LEW-Hermite},
Legendre \cite{LEW-Legendre}, Jacobi \cite{ELKYW-Jacobi}, and Laguerre
\cite{LW} equations. In these papers, the authors construct sequences - but
not the full continua - of left-definite spaces and left-definite operators
associated with the special self-adjoint operator $A$ that has the
corresponding classical orthogonal polynomials (of Jacobi, Legendre, Hermite,
Legendre, and Laguerre, respectively) as eigenfunctions. In \cite{LZ}, the
authors further applied this theory to the classical regular second-order
Fourier operator endowed with periodic boundary conditions. In this paper, we
extend the methods in \cite{LZ} to explicitly determine the sequences of
left-definite spaces $\{H_{n}\}_{n\in\mathbb{N}}$ and left-definite operators
$\{A_{n}\}_{n\in\mathbb{N}}$ associated with the Fourier operator $A$ in
$H=L^{2}[a,b]$ determined by semi-periodic boundary conditions. More
specifically, we consider the left-definite analysis for the self-adjoint
boundary value problem
\begin{equation}
\left\{
\begin{array}
[c]{c}%
\ell\lbrack y](x)=-y^{\prime\prime}(x)+ky(x)=\lambda y(x)\quad(x\in\lbrack
a,b])\\
y(a)=-y(b);\text{ }y^{\prime}(a)=-y^{\prime}(b);
\end{array}
\right.  \label{Fourier BVP}%
\end{equation}
here $[a,b]$ is a compact interval of the real line and $k$ is a fixed,
positive constant. As in the periodic case, the semi-periodic boundary value
problem is well-studied and important in various contexts in mathematics;
indeed, the eigenfunction expansion in the semi-periodic case produces a well
known Fourier series expansion for $f$ in $L^{2}[a,b].$ We extend this
expansion result to each of the left-definite spaces $H_{n}$ associated with
this self-adjoint boundary value problem$.$ Furthermore, as a consequence of
this analysis, for each positive integer $n,$ we obtain explicit
characterizations of the domains $\mathcal{D}(A^{n/2})$ of $A^{n/2}$ and of
the domains $\mathcal{D}(A_{n})$ of each of the left-definite operators
$A_{n}$ associated with $(H,A).$

As noted in the previously mentioned papers, the terminology
\textit{left-definite }is due to Sch\"{a}fke and Schneider \cite{SS} but the
origins of left-definite theory can be traced back to the work of Hermann Weyl
\cite{Weyl} in the early 1900's. The interest in left-definite theory
originated, at least in part, in the study of classical Sturm-Liouville
equations with a weight function that changes sign. There is a vast literature
for such left-definite problems; we refer to Zettl's text \cite[Chapters 5 and
15]{Zettl} and the recent text of Brown, Bennewitz, and Weikard \cite{BBW} for
excellent discussions of left-definite theory applied to second-order
Sturm-Liouville operators as well as references contained in these texts. We
note that when the general left-definite theory developed by Littlejohn and
Wellman is applied to (self-adjoint) Sturm-Liouville problems, we must assume
that these problems are right-definite; of course, this is not the case for
the left-definite theory laid out in \cite{BBW} or \cite{Zettl}. However,
except for \cite{LW}, the existing literature on left-definite Sturm-Liouville
theory, including that in \cite{BBW} and \cite{Zettl}, is limited to the study
of the \textit{first }left-definite setting (in the notation of \cite{LW}) and
does not discuss a continua of left-definite spaces or operators.

The most general self-adjoint operator $S,$ generated by the Fourier
expression $\ell\lbrack\cdot]$ in $L^{2}[a,b],$ with real coupled boundary
conditions is given by%
\begin{align}
Sf  &  =\ell\lbrack f]\nonumber\\
f\in\mathcal{D}(S)  &  =\left\{  f:[a,b]\rightarrow\mathbb{C}\mid f,f^{\prime
}\in AC[a,b];f^{\prime\prime}\in L^{2}[a,b];\left(
\begin{array}
[c]{r}%
f(a)\\
f^{\prime}(a)
\end{array}
\right)  =K\left(
\begin{array}
[c]{r}%
f(b)\\
f^{\prime}(b)
\end{array}
\right)  \right\}  , \label{General Coupled SA}%
\end{align}
where
\[
K=\left(
\begin{array}
[c]{rr}%
k_{1,1} & k_{1,2}\\
k_{2,1} & k_{2,2}%
\end{array}
\right)  \in SL_{2}(\mathbb{R});
\]
that is, $\det(K)=1.$ The periodic and semi-periodic boundary conditions are,
respectively, obtained from (\ref{General Coupled SA}), by letting $K=I$ or
$K=-I,$ where $I$ is the $2\times2$ identity matrix. These two special cases
are among the few cases where the eigenvalues and eigenfunctions are
explicitly computable. We are currently working on the general left-definite
theory for the general coupled self-adjoint boundary value problem presented
in (\ref{General Coupled SA}). The analysis involved in this general case is
proving to be considerably more difficult than that given in this manuscript
or \cite{LZ}. Indeed, for non-diagonal $K,$ additional analytic tools are
required to determine the sequences $\{H_{n}\}_{n=1}^{\infty}$ and
$\{A_{n}\}_{n=1}^{\infty}$ of the left-definite spaces and left-definite
operators associated with $(L^{2}[a,b],S).$ More specifically, methods
developed in this manuscript (and in \cite{LZ}) appear to only allow us to
determine $\{H_{2n}\}_{n=1}^{\infty}$ and $\{A_{2n}\}_{n=1}^{\infty}$ in the
general non-diagonal case of $K.$

The contents of this paper are as follows. In Section
\ref{Left-definite review}, we review the left-definite theory developed by
Littlejohn and Wellman. Section \ref{Section 3} deals with the semi-periodic
self-adjoint operator $A$ generated from (\ref{Fourier BVP}) and its
properties, including information about its spectrum, its eigenfunctions and
the fact that $A$ is bounded below in $L^{2}[a,b]$ by $kI,$ where $k$ is the
constant appearing in the differential expression in (\ref{Fourier BVP}). The
left-definite analysis of $A$ - specifically, the construction of the sequence
of left-definite spaces $\{H_{n}\}_{n\in\mathbb{N}}$ and left-definite
operators $\{A_{n}\}_{n\in\mathbb{N}}$ - is developed in Section
\ref{LD Spaces and Operators}. In addition, we develop a Fourier expansion
theorem valid in each left-definite space $H_{n}$ in Section
\ref{LD Spaces and Operators}. Lastly, in Section \ref{Concluding remarks},
some special cases of these left-definite spaces and left-definite operators
are discussed.

\section{A Review of Left-Definite Theory\label{Left-definite review}}

Let $V$ denote a vector space (over the complex field $\mathbb{C)}$ and
suppose that $(\cdot,\cdot)$ is an inner product with norm $\left\|
\cdot\right\|  $ generated from $(\cdot,\cdot)$ such that $H=(V,(\cdot
,\cdot))$ is a Hilbert space. Suppose $V_{r}$ (the subscripts will be made
clear shortly) is a linear manifold (vector subspace) of the vector space $V$
and let $(\cdot,\cdot)_{r}$ and $\left\|  \cdot\right\|  _{r}$ denote an inner
product and its associated norm, respectively, over $V_{r}$ (quite possibly
different from $(\cdot,\cdot)$ and $\left\|  \cdot\right\|  $). We denote the
resulting inner product space by $H_{r}=(V_{r},\left(  \cdot,\cdot\right)
_{r}).$

Throughout this section, we assume that $A:\mathcal{D}(A)\subset H\rightarrow
H$ is a self-adjoint operator that is bounded below by $kI$ for some $k>0;$
that is,
\[
(Ax,x)\geq k(x,x)\quad(x\in\mathcal{D}(A)).
\]
It follows that $A^{r},$ for each $r>0,$ is a self-adjoint operator that is
bounded below in $H$ by $k^{r}I.$

We now define an $r^{th}$ left-definite space associated with $(H,A).$

\begin{definition}
\label{Definition of rth LD space} Let $r>0$ and suppose $V_{r}$ is a linear
manifold of the Hilbert space $H$ $=(H,(\cdot,\cdot))$ and $\left(
\cdot,\cdot\right)  _{r}$ is an inner product on $V_{r}.$ Let $H_{r}%
=(V_{r},(\cdot,\cdot)_{r}).$ We say that $H_{r}$ is an $r^{th}$ left-definite
space\textbf{ }\textit{associated with the pair }$(H,A)$ if each of the
following conditions hold:\newline$(1)$ $H_{r}$ is a Hilbert space,\newline%
$(2)$ $\mathcal{D}(A^{r})$ is a linear manifold of $V_{r},$ \newline$(3)$
$\mathcal{D}(A^{r})$ is dense in $H_{r},$ \newline$(4)$ $\left(  x,x\right)
_{r}\geq k^{r}\left(  x,x\right)  \quad(x\in V_{r}),$ and\newline$(5)$
$\left(  x,y\right)  _{r}=\left(  A^{r}x,y\right)  \quad(x\in\mathcal{D}%
(A^{r}),\;y\in V_{r}).$
\end{definition}

It is not clear, from the definition, if such a self-adjoint operator $A$
generates a left-definite space for a given $r>0.$ However, in \cite{LW}, the
authors prove the following theorem; the Hilbert space spectral theorem plays
a critical role in establishing this result.

\begin{theorem}
\label{E/U LD Theorem} $($see \cite[Theorems 3.1 and 3.4]{LW}$)$ Suppose
$A:\mathcal{D}(A)\subset H\rightarrow H$ is a self-adjoint operator that is
bounded below by $kI,$ for some $k>0.$ For $r>0,$ define $H_{r}=(V_{r}%
,(\cdot,\cdot)_{r})$ by%
\begin{equation}
V_{r}=\mathcal{D}(A^{r/2}), \label{V_r}%
\end{equation}
and%
\[
(x,y)_{r}=(A^{r/2}x,A^{r/2}y)\quad(x,y\in V_{r}).
\]
Then $H_{r}$ is a left-definite space associated with the pair $(H,A).$
Moreover, suppose $H_{r}^{\prime}:=(V_{r}^{\prime},(\cdot,\cdot)_{r}^{\prime
})$ is another $r^{th}$ left-definite space associated with the pair $(H,A).$
Then $V_{r}=V_{r}^{\prime}$ and $(x,y)_{r}=(x,y)_{r}^{\prime}$ for all $x,y\in
V_{r}=V_{r}^{\prime};$ i.e. $H_{r}=H_{r}^{\prime}.$ That is to say,
$H_{r}=(V_{r},(\cdot,\cdot)_{r})$ is the \underline{unique} left-definite
space associated with $(H,A).$ Moreover,

\begin{enumerate}
\item[(a)] suppose $A$ is bounded. Then, for each $r>0,$

\begin{enumerate}
\item[(i)] $V=V_{r};$

\item[(ii)] the inner products $(\cdot,\cdot)$ and $(\cdot,\cdot)_{r}$ are equivalent.
\end{enumerate}

\item[(b)] suppose $A$ is unbounded. Then, for each $r,s>0,$

\begin{enumerate}
\item[(i)] $V_{r}$ is a proper subspace of $V;$

\item[(ii)] $V_{s}$ is a proper subspace of $V_{r}$ whenever $0<r<s;$

\item[(iii)] the inner products $(\cdot,\cdot)$ and $(\cdot,\cdot)_{r}$ are
not equivalent for any $r>0;$

\item[(iv)] the inner products $(\cdot,\cdot)_{r}$ and $(\cdot,\cdot)_{s}$ are
not equivalent for any $r,s>0,$ $r\neq s.$
\end{enumerate}
\end{enumerate}
\end{theorem}

\begin{remark}
\textbf{\noindent}\noindent\noindent\label{Importance of (5)} Although all
five conditions in Definition \ref{Definition of rth LD space} are necessary
in the proof of Theorem \ref{E/U LD Theorem}, the most important property, in
a sense, is the one given in (5). Indeed, this property asserts that the
$r^{th}$ left-definite inner product is generated from the $r^{th}$ power of
$A.$ If $A$ is generated from a Lagrangian symmetric differential expression
$\ell\lbrack\cdot],$ we see that the $r^{th}$ powers of $A$ are then
determined by the $r^{th}$ powers of $\ell\lbrack\cdot].$ Consequently, in
this case, it is possible to explicitly obtain these powers only when $r$ is a
positive integer. We refer the reader to \cite{LW} where, however, an example
of a self-adjoint operator $A$ in $\ell^{2}(\mathbb{N})$ is discussed in which
the \emph{entire} continuum of left-definite spaces is explicitly obtained.
\end{remark}

\begin{definition}
\label{Definition of LD operators} For $r>0,$ let $H_{r}=(V_{r},(\cdot
,\cdot)_{r})$ denote the $r^{th}$ left-definite space associated with $(H,A)$.
If there exists a self-adjoint operator $A_{r}:\mathcal{D}(A_{r})\subset$
$H_{r}\rightarrow H_{r}$ that is a restriction of $A;$ that is,
\begin{equation}
A_{r}f=Af\quad(f\in\mathcal{D}(A_{r})\subset\mathcal{D}(A)),\nonumber
\end{equation}
we call such an operator an\textbf{ }$r^{th}$\textbf{ }left-definite operator
associated with\textbf{ }$(H,A).$
\end{definition}

Again, it is not immediately clear that such an $A_{r}$ exists for a given
$r>0$; in fact, however, as the next theorem shows, $A_{r}$ exists and is
unique for each $r>0.$

\begin{theorem}
\label{Main Theorem2} $($see \cite[Theorems 3.2 and 3.4]{LW}$)$ Suppose $A$ is
a self-adjoint operator in a Hilbert space $H$ that is bounded below by $kI$
for some $k>0.$\ For any $r>0,$ let $H_{r}=(V_{r},(\cdot,\cdot)_{r})$ be the
$r^{th}$ left-definite space associated with $(H,A).$ Then there exists a
unique left-definite operator $A_{r}$ in $H_{r}$ associated with $(H,A);$ in
fact,
\[
\mathcal{D}(A_{r})=V_{r+2}.
\]
Each $A_{r}$ is bounded below in $H_{r}$ by $kI$. Moreover, from Theorem
\ref{E/U LD Theorem}, we have the following results:

\begin{enumerate}
\item[(a)] Suppose $A$ is bounded. Then, for each $r>0,$ $A=A_{r}.$

\item[(b)] Suppose $A$ is unbounded. Then, for each $r,s>0,$

\begin{enumerate}
\item[(i)] $\mathcal{D}(A_{r})$ is a proper subspace of $\mathcal{D}(A)$ for
each $r>0;$

\item[(ii)] $\mathcal{D}(A_{s})$ is a proper subspace of $\mathcal{D}(A_{r})$
whenever $0<r<s.$
\end{enumerate}
\end{enumerate}
\end{theorem}

The last theorem that we state in this section shows that the point spectrum,
continuous spectrum, and resolvent set of a self-adjoint operator $A$ and each
of its associated left-definite operators $A_{r}$ $(r>0)$ are identical.

\begin{theorem}
\label{Main Theorem3} $($see \cite[Theorem 3.6]{LW}$)$ For each $r>0,$ let
$A_{r}$ denote the $r^{th}$ left-definite operator associated with the
self-adjoint operator $A$ that is bounded below by $kI,$ where $k>0.$ Then
\end{theorem}

\begin{enumerate}
\item[(a)] \textit{the point spectra of }$A$\textit{ and }$A_{r}$\textit{
coincide; i.e. }$\sigma_{p}(A_{r})=\sigma_{p}(A);$

\item[(b)] \textit{the continuous spectra of }$A$\textit{ and }$A_{r}$\textit{
coincide; i.e. }$\sigma_{c}(A_{r})=\sigma_{c}(A);$

\item[(c)] \textit{the resolvent sets of }$A$\textit{ and }$A_{r}$\textit{ are
equal; i.e. }$\rho(A_{r})=\rho(A).$
\end{enumerate}

We refer the reader to \cite{LW} for additional theorems, and examples,
associated with the general left-definite theory of self-adjoint operators $A$
that are bounded below; see also \cite{LEW-Hermite}, \cite{LEW-Legendre},
\cite{ELKYW-Jacobi}, \cite{LW}, and \cite{LZ}.

\section{The Fourier Operator $A$ with semi-periodic boundary
conditions\label{Section 3}}

From here on, we let
\begin{equation}
H:=L^{2}[a,b], \label{H}%
\end{equation}
for $-\infty<a<b<\infty,$ denote the classical Hilbert space of Lebesgue
measurable functions $f:[a,b]\rightarrow\mathbb{C}$ satisfying $\int_{a}%
^{b}\left\vert f(x)\right\vert ^{2}dx<\infty$ with inner product%
\[
(f,g)_{L^{2}[a,b]}:=\int_{a}^{b}f(x)\overline{g}(x)dx\quad(f,g\in H),
\]
and associated norm%
\[
\left\Vert f\right\Vert _{L^{2}[a,b]}=(f,f)_{L^{2}[a,b]}^{1/2}\quad(f\in H).
\]
Fix $k>0$ and let $\ell\lbrack\cdot]$ denote the regular differential
(Fourier) expression defined by%
\begin{equation}
\ell\lbrack f](x):=-f^{\prime\prime}(x)+kf(x)\quad(x\in\lbrack a,b]).
\label{The Fourier expression}%
\end{equation}
The operator $A$ that we deal with in this paper is defined as%
\begin{equation}
\left\{
\begin{array}
[c]{l}%
Af=\ell\lbrack f]\quad(f\in\mathcal{D}(A))\\
\mathcal{D}(A)=\{f:[a,b]\rightarrow\mathbb{C}\mid f,f^{\prime}\in
AC[a,b];f^{\prime\prime}\in H;f(a)=-f(b);f^{\prime}(a)=-f^{\prime}(b)\}.
\end{array}
\right.  \label{Definition of A}%
\end{equation}
It is well known (see, for example, \cite{Naimark} or \cite{Zettl}) that $A$
is self-adjoint in $H$ and has a discrete spectrum $\sigma(A).$ A calculation
shows that the eigenvalues of $A$ are given by%
\begin{equation}
\lambda_{m}:=\left(  \frac{(2m-1)\pi}{b-a}\right)  ^{2}+k\quad(m\in
\mathbb{N}). \label{Eigenvalues of A}%
\end{equation}
For $m\in\mathbb{N},$ the general solution of $\ell\lbrack f](x)=\lambda
_{m}f(x)$ on $[a,b]$ is%
\[
f_{m}(x)=c_{m,1}\cos\left(  \frac{(2m-1)\pi}{b-a}x\right)  +c_{m,2}\sin\left(
\frac{(2m-1)\pi}{b-a}x\right)  .
\]
Furthermore, with%

\begin{equation}
\left\{
\begin{array}
[c]{c}%
y_{m,1}(x)=\cos\left(  \frac{(2m-1)\pi}{b-a}x\right)  \quad(m\in\mathbb{N})\\
y_{m,2}(x)=\sin\left(  \frac{(2m-1)\pi}{b-a}x\right)  \quad(m\in\mathbb{N}),
\end{array}
\right.  \label{Eigenfunctions for positive m}%
\end{equation}
calculations show that $\allowbreak$%
\begin{equation}
\left\Vert y_{m,j}\right\Vert _{L^{2}[a,b]}=\sqrt{\frac{b-a}{2}}\quad
(m\in\mathbb{N};\text{ }j=1,2). \label{Norms of eigenfunctions}%
\end{equation}
Consequently,
\begin{equation}
E=\{e_{m}\mid m\in\mathbb{N}\}=\{z_{m,1}\}_{m\in\mathbb{N}}\cup\{z_{m,2}%
\}_{m\in\mathbb{N}}, \label{Orthonormal basis in L^2[a,b]}%
\end{equation}
where%
\begin{align}
z_{m,1}(x)  &  =\sqrt{\dfrac{2}{b-a}}\cos\left(  \frac{(2m-1)\pi}%
{b-a}x\right)  \quad(m\in\mathbb{N})\label{Orthonormal vectors}\\
z_{m,2}(x)  &  =\sqrt{\dfrac{2}{b-a}}\sin\left(  \frac{(2m-1)\pi}%
{b-a}x\right)  \quad(m\in\mathbb{N}) \label{Orthonormal vectors2}%
\end{align}
is a complete orthonormal basis in $L^{2}[a,b].$

For later purposes, we observe that
\begin{equation}
z_{m,j}^{(r)}(a)=-z_{m,j}^{(r)}(b)\quad(j=1,2;r=0,1,\ldots).
\label{Derivatives of Eigenfunctions of A}%
\end{equation}
We remind the reader of the classical expansion theorem (see \cite[Chapter
4]{Rudin}) for functions $f\in L^{2}[a,b]$ in terms of the eigenfunctions of
$A$; that is, the classical Fourier series expansion theorem in $L^{2}[a,b]$.

\begin{theorem}
\label{Convergence of Fourier Series in L^[a,b]}Let $f\in L^{2}[a,b]$; for
each $N\in\mathbb{N}$, define the partial sums
\[
s_{N}(f)(x)=\sum_{m=1}^{N}a_{m}(f)\cos\left(  \frac{(2m-1)\pi}{b-a}x\right)
+\sum_{m=1}^{N}b_{m}(f)\sin\left(  \frac{(2m-1)\pi}{b-a}x\right)  \quad
(x\in\lbrack a,b]),
\]
where%
\begin{equation}
a_{m}(f):=(f,z_{m,1})=\sqrt{\frac{2}{b-a}}\int_{a}^{b}f(x)\cos\left(
\frac{(2m-1)\pi}{b-a}x\right)  dx\quad(m\in\mathbb{N}), \label{a_m(f)}%
\end{equation}
and%
\begin{equation}
b_{m}(f):=(f,z_{m,2})=\sqrt{\frac{2}{b-a}}\int_{a}^{b}f(x)\sin\left(
\frac{(2m-1)\pi}{b-a}x\right)  dx\quad(m\in\mathbb{N}) \label{b_m(f)}%
\end{equation}
are the Fourier coefficients of $f$ corresponding to the orthonormal basis $E$
defined in $($\ref{Orthonormal basis in L^2[a,b]}$).$ Then
\[
\left\Vert f-s_{N}(f)\right\Vert _{L^{2}[a,b]}\rightarrow0\text{ as
}N\rightarrow\infty,
\]
and%
\[
\left\Vert f\right\Vert _{L^{2}[a,b]}^{2}=\sum_{m=0}^{\infty}\left\vert
a_{m}(f)\right\vert ^{2}+\sum_{m=1}^{\infty}\left\vert b_{m}(f)\right\vert
^{2}.
\]

\end{theorem}

For $f\in\mathcal{D}(A),$ we see from integration by parts and the boundary
conditions in (\ref{Definition of A}) that%
\begin{align*}
(Af,f)_{L^{2}[a,b]}  &  =\int_{a}^{b}\left[  -f^{\prime\prime}%
(x)+kf(x)\right]  \overline{f}(x)dx\\
&  =-f^{\prime}(x)\overline{f}(x)\mid_{a}^{b}+\int_{a}^{b}\left[  \left\vert
f^{\prime}(x)\right\vert ^{2}+k\left\vert f(x)\right\vert ^{2}\right]  dx\\
&  =\int_{a}^{b}\left[  \left\vert f^{\prime}(x)\right\vert ^{2}+k\left\vert
f(x)\right\vert ^{2}\right]  dx\\
&  \geq k\int_{a}^{b}\left\vert f(x)\right\vert ^{2}dx=k(f,f)_{L^{2}[a,b]};
\end{align*}
that is, $A$ is bounded below by $kI$ in $H.$ Consequently, the left-definite
theory discussed in the last section can be applied to this operator $A.$ This
is done in the next section.

\section{The Left-Definite Spaces and Operators Associated with $(L^{2}%
[a,b],A)$\label{LD Spaces and Operators}}

Let the self-adjoint differential operator $A$ in $H=L^{2}[a,b]$ be defined by
(\ref{Definition of A}). In this section we use the theory given in Section
\ref{Left-definite review} to explicitly construct the left-definite spaces
$H_{n}$ and the left-definite operators $A_{n},$ associated with the pair
$(H,A),$ for all positive integer values of $n.$ As can be easily established
(see \cite{LZ}), we note for each $n\in\mathbb{N}$ that
\begin{equation}
\ell^{n}[y]=\ell\lbrack\ell^{n-1}[y]]=\sum_{j=0}^{n}(-1)^{j}\binom{n}%
{j}k^{n-j}y^{(2j)}. \label{Powers of Fourier expression}%
\end{equation}

\begin{definition}
\label{Definition of Left-Definite Spaces and Inner products}For
$n\in\mathbb{N},$ define

\begin{enumerate}
\item[(i)] $V_{n}:=\{f:[a,b]\rightarrow\mathbb{C\mid}f^{(j)}\in AC[a,b]$
$(j=0,1,\ldots,n-1);f^{(j)}(a)=-f^{(j)}(b)$ $(j=0,1,\ldots,n-1);f^{(n)}\in
L^{2}[a,b]\};$

\item[(ii)] $(f,g)_{n}:=\sum_{j=0}^{n}\binom{n}{j}k^{n-j}\int_{a}^{b}%
f^{(j)}(x)\overline{g}^{(j)}(x)dx\quad(f,g\in V_{n});$

\item[(iii)] $\left\|  f\right\|  _{n}:=(f,f)_{n}^{1/2};$

\item[(iv)] $H_{n}:=(V_{n},(\cdot,\cdot)_{n}).$
\end{enumerate}
\end{definition}

\begin{remark}
We note that $V_{n}$ is a vector subspace of $L^{2}[a,b]$. Furthermore, since
$k>0,$ it is clear that $(\cdot,\cdot)_{n}$ is an inner product on
$V_{n}\times V_{n}$.
\end{remark}

\begin{remark}
Notice that the inner product $(\cdot,\cdot)_{n}$ is generated by the $n^{th}$
integral power $\ell^{n}[\cdot]$ of the differential expression $\ell
\lbrack\cdot];$ indeed, see $($\ref{Powers of Fourier expression}$)$ and item
$(5)$ in Definition
\ref{Definition of Left-Definite Spaces and Inner products}.
\end{remark}

\begin{remark}
\label{Eigenfunctions in H_n}From $($\ref{Derivatives of Eigenfunctions of A}%
$),$ we see that $E,$ the set of orthonormal eigenfunctions of $A$ given in
$($\ref{Orthonormal basis in L^2[a,b]}$),$ is contained in $H_{n}$ for each
$n\in\mathbb{N}.$ In Theorem 4.5 below we show that $E$ is a complete
orthogonal set in each space $H_{n}.$ Theorem \ref{H_n - LD space} shows that
$H_{n}$ is the $n^{th}$ left-definite space associated with the pair $(H,A).$
\end{remark}

\begin{theorem}
\label{Completeness of H_n}For each $n\in\mathbb{N},$ the space $H_{n},$
defined in Definition
\ref{Definition of Left-Definite Spaces and Inner products}, is a Hilbert space.
\end{theorem}

\begin{proof}
Suppose $\{f_{m}\}\subset H_{n}$ is Cauchy so%

\begin{equation}%
\begin{split}
\Vert f_{m}-f_{r}\Vert_{n}^{2}  &  =(f_{m}-f_{r},f_{m}-f_{r})_{n}^{2}\\
&  =\sum_{j=0}^{n}{\binom{n}{j}}k^{n-j}\int_{a}^{b}|f_{m}^{(n-j)}%
(t)-f_{r}^{(n-j)}(t)|^{2}\,dt\\
&  =\sum_{j=0}^{n}{\binom{n}{j}}k^{n-j}\Vert f_{m}^{(j)}-f_{r}^{(j)}%
\Vert_{L^{2}[a,b]}^{2}\\
&  \rightarrow0\text{\quad as }m,r\rightarrow\infty.
\end{split}
\label{eq1}%
\end{equation}
For $0\leq j\leq n,$ $\left\Vert f_{m}^{(j)}-f_{r}^{(j)}\right\Vert
_{L^{2}[a,b]}^{2}\leq\left\Vert f_{m}-f_{r}\right\Vert _{n}^{2}$ so
$\{f_{m}^{(j)}\}$ is Cauchy in $L^{2}[a,b].$ From the completeness of
$L^{2}[a,b],$ there exists $g_{j}\in L^{2}[a,b]$ such that%
\begin{equation}
f_{m}^{(j)}\rightarrow g_{n-j}\text{ in }L^{2}[a,b]\quad(j=0,1,\ldots,n).
\label{Completeness-1}%
\end{equation}
In particular,
\[
f_{m}^{(n)}\rightarrow g_{0}\text{ in }L^{2}[a,b].
\]
By H\"{o}lder's inequality, $f_{m}^{(n)}\rightarrow g_{0}$ in $L^{1}[a,b];$
moreover,
\begin{equation}
f_{m}^{(n-1)}(x)-f_{m}^{(n-1)}(a)=\int_{a}^{x}f_{m}^{(n)}(t)dt\rightarrow
\int_{a}^{x}g_{0}(t)dt\quad(x\in\lbrack a,b]) \label{Completeness-2}%
\end{equation}
and, in particular,
\[
f_{m}^{(n-1)}(b)-f_{m}^{(n-1)}(a)=\int_{a}^{b}f_{m}^{(n)}(t)dt\rightarrow
\int_{a}^{b}g_{0}(t)dt.
\]
Since $f_{m}^{(n-1)}(b)=-f_{m}^{(n-1)}(a),$ the above identity is rewritten as
$-2f_{m}^{(n-1)}(a)\rightarrow\int_{a}^{b}g_{0}(t)dt$ so that%
\begin{equation}
f_{m}^{(n-1)}(a)\rightarrow-\frac{1}{2}\int_{a}^{b}g_{0}(t)dt:=A_{0}.
\label{Completeness-3}%
\end{equation}
Combining (\ref{Completeness-2})\ and (\ref{Completeness-3}), we see that%
\begin{equation}
f_{m}^{(n-1)}(x)\rightarrow A_{0}+\int_{a}^{x}g_{0}(t)dt\quad(x\in\lbrack
a,b]). \label{Completeness-4}%
\end{equation}
From (\ref{Completeness-1}), with $j=n-1,$ we find
\begin{equation}
f_{m}^{(n-1)}\rightarrow g_{1}\text{ in }L^{2}[a,b]. \label{Completeness-5}%
\end{equation}
Combining (\ref{Completeness-4}) and (\ref{Completeness-5}), we see that%
\begin{equation}
g_{1}(x)=A_{0}+\int_{a}^{x}g_{0}(t)dt\quad(x\in\lbrack a,b]).
\label{Completeness-6}%
\end{equation}
Observe that

\begin{enumerate}
\item[(i)] $g_{1}\in AC[a,b],$

\item[(ii)] $g^{\prime}=g_{0},$

\item[(iii)] $g_{1}(a)=A_{0},$

\item[(iv)] $g_{1}(b)=A_{0}+\int_{a}^{b}g_{0}(t)dt=A_{0}-2A_{0}=-A_{0}%
=-g_{1}(a).$
\end{enumerate}

It is helpful to consider the next iteration in the construction of the limit
function $f$ of the Cauchy sequence $\{f_{m}\}$ in $H_{2}.$ Repeating the
above analysis, we see that%
\begin{equation}
g_{2}(x)=A_{1}+\int_{a}^{x}g_{1}(t)dt, \label{Completeness-7}%
\end{equation}
where%
\begin{equation}
A_{1}=-\frac{1}{2}\int_{a}^{b}g_{1}(t)dt. \label{Completeness-8}%
\end{equation}
From (\ref{Completeness-7}), (\ref{Completeness-8}), and the above
calculations, we see that

\begin{enumerate}
\item[(I)] $g_{2}\in AC[a,b]$

\item[(II)] $g_{2}^{\prime}=g_{1}\in AC[a,b]$

\item[(III)] $g_{2}^{\prime\prime}=g_{1}^{\prime}=g_{0}$

\item[(IV)] $g_{2}(a)=A_{1}$

\item[(V)] $g_{2}(b)=A_{1}+\int_{a}^{b}g_{1}(t)dt=-A_{1}=-g_{2}(a);$

\item[(VI)] $g_{2}^{\prime}(b)=g_{1}(b)=-g_{1}(a)=-g_{2}^{\prime}(a).$
\end{enumerate}

\noindent Repeating this argument $i$ times, where $1\leq i\leq n,$ the above
procedure shows that

\begin{enumerate}
\item[(a)] $g_{i}(x)=A_{i-1}+\int_{a}^{x}g_{i-1}(t)dt$

\item[(b)] $g_{i}\in AC[a,b]$

\item[(c)] $g_{i}^{\prime}=g_{i-1},$ $g_{i}^{\prime\prime}=g_{i-2}%
,\ldots,g_{i}^{(i)}=g_{0}$

\item[(d)] $g_{i},g_{i-1}^{\prime},\ldots,g_{i}^{(i-1)}\in AC[a,b]$

\item[(e)] $g_{i}(b)=-g_{i}(a);$ $g_{i}^{\prime}(b)=-g_{i}^{\prime}%
(a);\ldots,g_{i}^{(i-1)}(b)=-g_{i}^{(i-1)}(a).$
\end{enumerate}

\noindent Choosing $i=n$ above, we see that the function $f:=g_{n}\in H_{n}.$
Moreover, from (a)-(e) and (\ref{Completeness-1}), we see that%
\begin{align*}
\left\Vert f_{m}-f\right\Vert _{n}^{2}  &  =\sum_{j=0}^{n}\binom{n}{j}%
k^{n-j}\left\Vert f_{m}^{(j)}-f^{(j)}\right\Vert _{L^{2}[a,b]}^{2}\\
&  =\sum_{j=0}^{n}\binom{n}{j}k^{n-j}\left\Vert f_{m}^{(j)}-g_{n-j}\right\Vert
_{L^{2}[a,b]}^{2}\\
&  \rightarrow0\text{ as }m\rightarrow\infty.
\end{align*}
Thus, $H_{n}=(V_{n},(\cdot,\cdot)_{n})$ is a Hilbert space.
\end{proof}

Observe, for $f\in H_{n},$%
\begin{equation}
\left\Vert f\right\Vert _{n}\geq k^{n}\left\Vert f\right\Vert _{L^{2}[a,b]}.
\label{Inequality}%
\end{equation}

Recall that, for $j=1,2$ and $m\in\mathbb{N},$ the trigonometric functions
$z_{m,j},$ defined in (\ref{Orthonormal vectors}) and
(\ref{Orthonormal vectors2}), are orthonormal eigenfunctions in $L^{2}[a,b]$
associated with the eigenvalue $\lambda_{m}$ defined in
(\ref{Eigenvalues of A}). From (\ref{Derivatives of Eigenfunctions of A}),
(\ref{Powers of Fourier expression}), integration by parts, and the definition
of $V_{n},$ we see that, for $j=1,2$ and $f\in H_{n},$%

\begin{equation}%
\begin{split}
\lambda_{m}^{n}(z_{m,j},f)_{L^{2}[a,b]}  &  =(A^{n}z_{m,j},f)_{L^{2}%
[a,b]}=\int_{a}^{b}\ell^{n}[z_{m,j}](x)\overline{f}(x)\,dx\\
&  =\sum_{r=0}^{n}(-1)^{r}{\binom{n}{r}}k^{n-r}\int_{a}^{b}z_{m,j}%
^{(2r)}(x)\bar{f}(x)\,dx\\
&  =\sum_{r=0}^{n}(-1)^{r}{\binom{n}{r}}k^{n-r}\left[  \left.  \sum
_{s=0}^{r-1}(-1)^{s}z_{m,j}^{(2r-1-s)}(x)\overline{f}^{(s)}(x)\right\vert
_{a}^{b}+(-1)^{r}\int_{a}^{b}z_{m,j}^{(r)}(x)\bar{f}^{(r)}(x)\,dx\right] \\
&  =\sum_{r=0}^{n}\binom{n}{r}k^{n-r}\int_{a}^{b}z_{m,j}^{(r)}(x)\overline
{f}^{(r)}(x)dx\\
&  =(z_{m,j},f)_{n}.
\end{split}
\label{Fundamental Relation}%
\end{equation}

In particular, if we set $f=z_{r,i}$ in (\ref{Fundamental Relation}), we see
that%
\begin{equation}
(z_{m,j},z_{r,i})_{n}=\lambda_{m}^{n}(z_{m,j},z_{r,i})_{L^{2}[a,b]}%
=\lambda_{m}^{n}\delta_{m,r}\quad(i,j=1,2;m,r\in\mathbb{N}).
\label{Fundamental Relation 2}%
\end{equation}
From Theorem \ref{Convergence of Fourier Series in L^[a,b]}, the
orthonormality in $L^{2}[a,b]$ of the functions $E=\{z_{m,1}\}_{m=1}^{\infty
}\cup\{z_{m,2}\}_{m=1}^{\infty}$, and (\ref{Fundamental Relation 2}), we
obtain the following theorem.

\begin{theorem}
\label{Orthogonality of E in H_n}For each $n\in\mathbb{N},$ the set
\begin{equation}
E_{n}:=\{Z_{m,n,1}\}_{m\in\mathbb{N}}\cup\{Z_{m,n,2}\}_{m\in\mathbb{N}},
\label{E_n}%
\end{equation}
where%
\begin{equation}
\left\{
\begin{array}
[c]{l}%
Z_{m,n,1}(x)=\sqrt{\dfrac{2}{b-a}}\dfrac{1}{\sqrt{\left(  \left(
\frac{(2m-1)\pi}{b-a}\right)  ^{2}+k\right)  ^{n}}}\cos\left(  \frac
{(2m-1)\pi}{b-a}x\right)  =\lambda_{m}^{-n/2}z_{m,1}(x)\\
Z_{m,n,2}(x)=\sqrt{\dfrac{2}{b-a}}\dfrac{1}{\sqrt{\left(  \left(
\frac{(2m-1)\pi}{b-a}\right)  ^{2}+k\right)  ^{n}}}\sin\left(  \frac
{(2m-1)\pi}{b-a}x\right)  =\lambda_{m}^{-n/2}z_{m,2}(x)
\end{array}
\right.  \label{Z_1 and Z_2}%
\end{equation}
forms an orthonormal set in $H_{n}.$
\end{theorem}

Later in this section (see Theorem \ref{H_n - LD space}, part (3)), we prove
that $E_{n}$ is, in fact, a \textit{complete} orthonormal set in $H_{n}$ for
each $n\in\mathbb{N}.$

For later purposes, we need the following equality involving finite linear
combinations of eigenfunctions of $A$ - the so-called \textit{trigonometric
polynomials}$.$ Let $N_{1},M_{1},N,M\in\mathbb{N}$ with $N_{1}\leq N$ and
$M_{1}\leq M$ and let $\alpha_{m},\beta_{r}\in\mathbb{C}$ $(m=N_{1},\ldots,N;$
$r=M_{1},\ldots,M).$ Suppose%
\[
p(x)=\sum_{m=N_{1}}^{N}\alpha_{m}e_{m}(x),\text{ }q(x)=\sum_{r=M_{1}}^{M}%
\beta_{r}e_{r}(x),
\]
where each $e_{m}\in E$, defined in (\ref{Orthonormal basis in L^2[a,b]}).
Then $p,q\in H_{n}$ for all $n\in\mathbb{N}$ and, by
(\ref{Fundamental Relation}) and linearity, we see that%
\begin{align}
(A^{n}p,q)_{L^{2}[a,b]}  &  =\sum_{m=N_{1}}^{N}\sum_{r=M_{1}}^{M}\alpha
_{m}\overline{\beta}_{r}(A^{n}e_{m},e_{r})_{L^{2}[a,b]}\nonumber\\
&  =\sum_{m=N_{1}}^{N}\sum_{r=M_{1}}^{M}\alpha_{m}\overline{\beta}_{r}%
(e_{m},e_{r})_{n}\text{ }\label{Trigonometric Polynomial Identity}\\
&  =(\sum_{m=N_{1}}^{N}\alpha_{m}e_{m},\sum_{r=M_{1}}^{M}\beta_{r}e_{r}%
)_{n}\nonumber\\
&  =(p,q)_{n}.\nonumber
\end{align}
We are now in position to prove the following main theorem.

\begin{theorem}
\label{H_n - LD space}For each $n\in\mathbb{N},$ let
\begin{equation}
H_{n}=(V_{n},(\cdot,\cdot)_{n}), \label{H_n definition}%
\end{equation}
where%
\begin{equation}
V_{n}:=\{f:[a,b]\rightarrow\mathbb{C}\mid f^{(j)}\in AC[a,b],\text{ }%
f^{(j)}(a)=-f^{(j)}(b)\text{ }(j=0,1,\ldots,n-1);f^{(n)}\in L^{2}[a,b]\},
\label{V_n definition}%
\end{equation}
and
\begin{equation}
(f,g)_{n}:=\sum_{j=0}^{n}\binom{n}{j}k^{n-j}\int_{a}^{b}f^{(j)}(x)\overline
{g}^{(j)}(x)dx\quad(f,g\in V_{n}). \label{IP definition}%
\end{equation}
Then $H_{n}$ is the $n^{th}$ left-definite space associated with the pair
$(H,A).$
\end{theorem}

\begin{proof}
Let $n\in\mathbb{N}.$ We need to establish properties (1)-(5) in Definition
\ref{Definition of rth LD space}.\medskip\newline(i) \underline{$H_{n}$ is a
Hilbert space}\newline This is proved in Theorem \ref{Completeness of H_n}%
.\medskip\newline(ii) \underline{$\mathcal{D}(A^{n})\subset V_{n}$}\newline
Let $f\in\mathcal{D}(A^{n}).$ Since the set $E=\{e_{m}\mid m\in\mathbb{N}%
\}=\{z_{m,1}\}_{m=1}^{\infty}\cup\{z_{m,2}\}_{m=1}^{\infty}$ of eigenfunctions
of $A$ form a complete orthonormal set in $L^{2}[a,b],$ we see that%
\begin{equation}
p_{j}:=\sum_{m=0}^{j}c_{m}e_{m}\rightarrow f\text{ as }j\rightarrow
\infty\text{ in }L^{2}[a,b], \label{Main-0}%
\end{equation}
where $\{c_{m}\}$ are the Fourier coefficients of $f$ in $L^{2}[a,b],$ defined
by%
\[
c_{m}:=\int_{a}^{b}f(t)e_{m}(t)dt=(f,e_{m})_{L^{2}[a,b]}\quad(m\in
\mathbb{N}_{0}).
\]
Since $A^{n}f\in L^{2}[a,b],$ we also have%
\begin{equation}
\sum_{m=0}^{j}d_{m}e_{m}\rightarrow A^{n}f\text{ as }j\rightarrow\infty\text{
in }L^{2}[a,b], \label{Main-1}%
\end{equation}
where%
\[
d_{m}=(A^{n}f,e_{m})_{L^{2}[a,b]}\quad(m\in\mathbb{N}).
\]
With $\widetilde{\lambda}_{m}$ denoting the eigenvalue of $A$ associated with
$e_{m},$ we see from the self-adjointness of $A$ that%
\[
d_{m}=(A^{n}f,e_{m})_{L^{2}[a,b]}=(f,A^{n}e_{m})_{L^{2}[a,b]}%
=\widetilde{\lambda}_{m}^{n}(f,e_{m})_{L^{2}[a,b]}=\widetilde{\lambda}_{m}%
^{n}c_{m}.
\]
Substituting this identity into (\ref{Main-1}), and using the linearity of
$A^{n},$ we obtain%
\begin{equation}
A^{n}p_{j}\rightarrow A^{n}f\text{ as }j\rightarrow\infty\text{ in }%
L^{2}[a,b], \label{Main-2}%
\end{equation}
where $p_{j}$ is defined in (\ref{Main-0}). From
(\ref{Trigonometric Polynomial Identity}), (\ref{Main-0}), and (\ref{Main-2}),
it follows that%
\begin{align*}
\left\Vert p_{j}-p_{r}\right\Vert _{n}^{2}  &  =(A^{n}(p_{j}-p_{r}%
),p_{j}-p_{r})_{L^{2}[a,b]}\\
&  \rightarrow0\text{ as }j,r\rightarrow\infty;
\end{align*}
that is to say, $\{p_{j}\}_{j\in\mathbb{N}}$ is Cauchy in $H_{n}.$ From the
completeness of $H_{n},$ there exists $g\in V_{n}\subset L^{2}[a,b]$ such that%
\[
p_{j}\rightarrow g\text{ in }H_{n}.
\]
Furthermore, (\ref{Inequality}) shows us that
\[
\left\Vert p_{j}-g\right\Vert _{n}^{2}\geq k^{n}\left\Vert p_{j}-g\right\Vert
_{L^{2}[a,b]}^{2},
\]
so%
\begin{equation}
p_{j}\rightarrow g\text{ as }j\rightarrow\infty\text{ in }L^{2}[a,b].
\label{Main-3}%
\end{equation}
Comparing (\ref{Main-0}) and (\ref{Main-3}), we see that $f=g\in V_{n};$
consequently, $\mathcal{D}(A^{n})\subset V_{n}$ as required.\medskip
\newline(iii) \underline{$\mathcal{D}(A^{n})$ is dense in $H_{n}$}\newline
Since $E$ is contained in $\mathcal{D}(A^{n}),$ it suffices to show that $E$
is a complete orthogonal set in $H_{n}.$ From this, it will follow (see
\cite[Chapter 4]{Rudin}) that the vector subspace $T$ $\subset\mathcal{D}%
(A^{n})$ of all trigonometric polynomials (that is, all finite linear
combinations of elements from the set $E$ is dense in $H_{n}$ and,
consequently, $\mathcal{D}(A^{n})$ is dense in $H_{n}.$ To this end, suppose%
\[
(e_{m},f)_{n}=0\quad(m\in\mathbb{N}_{0})
\]
for some $f\in H_{n}.$ From (\ref{Fundamental Relation}), we see that%
\[
0=(e_{m},f)_{n}=(A^{n}e_{m},f)_{L^{2}[a,b]}=\widetilde{\lambda}_{m}^{n}%
(e_{m},f)_{L^{2}[a,b]},
\]
where $\widetilde{\lambda}_{m}>0$ is the eigenvalue associated with $e_{m}.$
It follows that%
\begin{equation}
(e_{m},f)_{L^{2}[a,b]}=0\quad(m\in\mathbb{N}_{0}). \label{Main-4}%
\end{equation}
As remarked in Section \ref{Section 3}, $E$ is a complete orthonormal set in
$L^{2}[a,b];$ consequently, (\ref{Main-4}) implies that $f=0$ in $L^{2}[a,b].$
From this, it is clear that $f=0$ in $H_{n},$ thereby completing the proof
that $E$ is a complete orthogonal set in $H_{n}.$ Consequently, we see that
$E_{n},$ defined in (\ref{E_n}), is a complete orthonormal set in
$H_{n}.\medskip$\newline(iv) \underline{$(f,f)_{n}\geq k^{n}(f,f)_{L^{2}%
[a,b]}$ for all $f\in V_{n}$}\newline This is clear from the definition of
$(\cdot,\cdot)_{n}:$%
\begin{align*}
(f,f)_{n}  &  =\sum_{j=0}^{n}\binom{n}{j}k^{n-j}\int_{a}^{b}\left\vert
f^{(j)}(x)\right\vert ^{2}dx\\
&  \geq k^{n}\int_{a}^{b}\left\vert f^{(j)}(x)\right\vert ^{2}dx=k^{n}%
(f,f)_{L^{2}[a,b]};
\end{align*}
\medskip\newline see also (\ref{Inequality}).$\medskip$\newline(v)
\underline{$(A^{n}f,g)_{L^{2}[a,b]}=(f,g)_{n}$ for all $f\in\mathcal{D}%
(A^{n})$ and $g\in V_{n}$}\newline Let $f\in\mathcal{D}(A^{n})$ and $g\in
V_{n}.$ From (\ref{Trigonometric Polynomial Identity}), we see that%
\begin{equation}
(A^{n}p,q)_{L^{2}[a,b]}=(p,q)_{n} \label{Main-5}%
\end{equation}
for all trigonometric polynomials $p$ and $q$ of the form%
\[
p=\sum_{m=1}^{N}\alpha_{m}e_{m},\text{ }q=\sum_{m=1}^{M}\beta_{m}e_{m}.
\]
From part (iii) of this proof, we know that the space $T$ of all trigonometric
polynomials is dense in $H_{n}.$ Hence there exists $\{p_{j}\}_{j\in
\mathbb{N}},\{q_{j}\}_{j\in\mathbb{N}}\subset T$ such that%
\begin{equation}
p_{j}\rightarrow f,\text{ }q_{j}\rightarrow g\text{ as }j\rightarrow
\infty\text{ in }H_{n}. \label{Main-6}%
\end{equation}
Since convergence in $H_{n}$ implies convergence in $L^{2}[a,b]$ (from part
(iii))), we see that%
\begin{equation}
p_{j}\rightarrow f,\text{ }q_{j}\rightarrow g\text{ as }j\rightarrow
\infty\text{ in }L^{2}[a,b]. \label{Main-7}%
\end{equation}
Moreover, from part (ii) of this proof, we see that%
\begin{equation}
A^{n}p_{j}\rightarrow A^{n}f\text{ as }j\rightarrow\infty\text{ in }%
L^{2}[a,b]. \label{Main-8}%
\end{equation}
Consequently, from (\ref{Main-5}), (\ref{Main-6}), (\ref{Main-7}), and
(\ref{Main-8}), we see that%
\[
(A^{n}f,g)_{L^{2}[a,b]}=\lim_{j\rightarrow\infty}(A^{n}p_{j},q_{j}%
)_{L^{2}[a,b]}=\lim_{j\rightarrow\infty}(p_{j},q_{j})_{n}=(f,g)_{n}.
\]
This completes the proof of (v) and the proof of the theorem.
\end{proof}

The following result, part of which is proved in step (iii) of the above
theorem, is the analogous result in each left-definite space $H_{n}$ of the
classical Fourier expansion theorem in $L^{2}[a,b]\ $stated in Theorem
\ref{Convergence of Fourier Series in L^[a,b]}. Note the identities in
(\ref{Fourier rel 1}) and (\ref{Fourier rel 2}); these formulae relate the
Fourier coefficients of $f$ relative to the orthonormal basis $E_{n}$ of
$H_{n}\subset L^{2}[a,b]$ to the Fourier coefficients of $f$ relative to the
orthonormal basis $E$ of $L^{2}[a,b]$.

\begin{theorem}
\label{Completeness of E in each H_n} $($Fourier Expansion Theorem in
Left-Definite Spaces$)$ For each $n\in\mathbb{N},$ let%
\[
E_{n}=\{Z_{m,n,1}\}_{m\in\mathbb{N}}\cup\{Z_{m,n,2}\}_{m\in\mathbb{N}}%
\]
be as in $($\ref{E_n}$)$ and $($\ref{Z_1 and Z_2}$).$ Then $E_{n}$ is a
complete orthonormal set in $H_{n}.$ Furthermore, for $f\in H_{n}\subset
L^{2}[a,b]$ and $N\in\mathbb{N},$ define the partial sums%
\[
S_{N,n}(f)(x)=\sum_{m=1}^{N}A_{m,n}(f)\cos\left(  \frac{(2m-1)\pi}%
{b-a}x\right)  +\sum_{m=1}^{N}B_{m,n}(f)\sin\left(  \frac{(2m-1)\pi}%
{b-a}x\right)  \quad(x\in\lbrack a,b]),
\]
where $\{A_{m,n}(f)\}_{m\in\mathbb{N}}$ and $\{B_{m,n}(f)\}_{m\in\mathbb{N}}$
are the Fourier coefficients of $f$ relative to $E_{n}$ defined by%
\begin{equation}
A_{m,n}(f):=(f,Z_{m,n,1})_{n}\quad(m\in\mathbb{N}) \label{A_(m,n)}%
\end{equation}
and%
\begin{equation}
B_{m,n}(f):=(f,Z_{m,n,2})_{n}\quad(m\in\mathbb{N}). \label{B_(m,n)}%
\end{equation}
Then

\begin{enumerate}
\item[(a)] $\left\|  f-S_{N,n}(f)\right\|  _{n}\rightarrow0$ as $N\rightarrow
\infty;$

\item[(b)] $\left\Vert f\right\Vert _{n}^{2}=\sum_{m=0}^{\infty}\left\vert
A_{m,n}(f)\right\vert ^{2}+\sum_{m=1}^{\infty}\left\vert B_{m,n}(f)\right\vert
^{2};$

\item[(c)]
\begin{align}
A_{m,n}(f)  &  =\lambda_{m}^{n/2}a_{m}(f)\quad(m\in\mathbb{N}%
)\label{Fourier rel 1}\\
B_{m,n}(f)  &  =\lambda_{m}^{n/2}b_{m}(f)\quad(m\in\mathbb{N}),
\label{Fourier rel 2}%
\end{align}
where $\{a_{m}(f)\}_{m\in\mathbb{N}}$ and $\{b_{m}(f)\}_{m\in\mathbb{N}}$ are
the Fourier coefficients of $f$, defined respectively in $($\ref{a_m(f)}$),$
and $($\ref{b_m(f)}$),$ relative to the orthonormal basis $E,$ given in
$($\ref{Orthonormal basis in L^2[a,b]}$),$ in $L^{2}[a,b]$ and where
$\{\lambda_{m}\}_{m\in\mathbb{N}}$ are the eigenvalues of $A$ defined in
$($\ref{Eigenvalues of A}$).$
\end{enumerate}
\end{theorem}

\begin{proof}
The fact that $E_{n}$ is a complete orthonormal set in $H_{n}$ is given in the
proof of part (iii) in Theorem \ref{H_n - LD space}. Let $f\in H_{n}.$ The
proofs of parts (a) and (b) are standard for any complete orthonormal set in a
Hilbert space; see \cite[Theorem 4.18]{Rudin}. For $m\in\mathbb{N},$ we see
from (\ref{Fundamental Relation 2}) that
\begin{align}
A_{m,n}(f)  &  =(f,Z_{m,n,1})_{n}\label{Calculation of A_(m,n)}\\
&  =(f,\lambda_{m}^{-n/2}z_{m,1})_{n}=\lambda_{m}^{-n/2}(f,z_{m,1}%
)_{n}\nonumber\\
&  =\lambda_{m}^{-n/2}(f,A^{n}z_{m,1})_{L^{2}[a,b]}\nonumber\\
&  =\lambda_{m}^{-n/2}\lambda_{m}^{n}(f,z_{m,1})_{L^{2}[a,b]}\nonumber\\
&  =\lambda_{n}^{n/2}(f,z_{m,1})_{L^{2}[a,b]}\nonumber\\
&  =\lambda_{n}^{n/2}a_{m}(f);\nonumber
\end{align}
this proves (\ref{Fourier rel 1}). A similar calculation establishes
(\ref{Fourier rel 2}).
\end{proof}

By combining Theorem \ref{H_n - LD space} with Theorems \ref{Main Theorem2}
and \ref{Main Theorem3}, we obtain the following result concerning the
sequence of left-definite operators $\{A_{n}\}_{n\in\mathbb{N}}$ associated
with the pair $(H,A).$

\begin{theorem}
\label{Sequence of LD operators}Let $n\in\mathbb{N}$ and let $H_{n}%
=(V_{n},(\cdot,\cdot)_{n})$ be the $n^{th}$ left-definite operator associated
with the pair $(L^{2}[a,b],A).$ Define the operator $A_{n}:\mathcal{D}%
(A_{n})\subset H_{n}\rightarrow H_{n}$ by%
\begin{align*}
\mathcal{D}(A_{n})  &  :=V_{n+2}\\
&  =\{f:[a,b]\rightarrow\mathbb{C}\mid f^{(j)}\in AC[a,b],\text{ }%
f^{(j)}(a)=-f^{(j)}(b)\text{ }(j=0,1,\ldots,n+1);f^{(n+2)}\in L^{2}[a,b]\}\\
(A_{n}f)(x)  &  :=\ell\lbrack f](x)=-y^{\prime\prime}(x)+ky(x)\quad
(f\in\mathcal{D}(A_{n})).
\end{align*}
Then $A_{n}$ is the $n^{th}$ left-definite operator associated with the pair
$(H,A).$ In particular, $A_{n}$ is self-adjoint in $H_{n}$ and the spectrum
$\sigma(A_{n})$ is a purely discrete point spectrum given explicitly by%
\[
\sigma(A_{n})=\sigma(A)=\left\{  \left(  \frac{(2m-1)\pi}{b-a}\right)
^{2}+k\mid m\in\mathbb{N}\right\}  .
\]

\end{theorem}

\section{Concluding Remarks\label{Concluding remarks}}

In this last section, we focus on some special left-definite spaces and
operators for the semi-periodic operator $A.$

\begin{remark}
\label{Square root domain}For an arbitrary self-adjoint operator $A$ in a
Hilbert space that is bounded below by a positive constant we see, from
Theorem \ref{E/U LD Theorem}, that the domain $\mathcal{D}(A^{1/2})$ of its
positive square root $A^{1/2}$ is given by the first left-definite vector
space $V_{1}.$ For our specific operator $A,$ defined in $($%
\ref{Definition of A}$),$ we have the explicit characterization of this
domain:%
\begin{equation}
\mathcal{D}(A^{1/2})=\{f:[a,b]\rightarrow\mathbb{C\mid}f\in
AC[a,b];f(a)=-f(b);f^{\prime}\in L^{2}[a,b]\}.
\label{Square root operator domain}%
\end{equation}

\end{remark}

\begin{remark}
From Theorem \ref{Sequence of LD operators}, the domain of the first
left-definite operator $A_{1},$ which is a self-adjoint operator in the first
left-definite space $H_{1},$ is given by%
\[
\mathcal{D}(A_{1})=V_{3}=\{f:[a,b]\rightarrow\mathbb{C}\mid f^{(j)}\in
AC[a,b]\text{ and }f^{(j)}(a)=-f^{(j)}(b)\text{ }(j=0,1,2);\;f^{(3)}\in
L^{2}[a,b]\}.
\]
Notice that $\mathcal{D}(A_{1})$ is also the domain of $A^{3/2}.$ The domain
of the second left-definite operator $A_{2},$ which is self-adjoint in the
Hilbert space $H_{2}=(V_{2},(\cdot,\cdot)_{2}),$ where%
\begin{align*}
V_{2}  &  =\{f:[a,b]\rightarrow\mathbb{C}\mid f^{(j)}\in AC[a,b]\text{ and
}f^{(j)}(a)=-f^{(j)}(b)\text{ }(j=0,1);\;f^{\prime\prime}\in L^{2}[a,b]\}\\
(f,g)_{2}  &  =\int_{a}^{b}\left(  f^{\prime\prime}(x)\overline{g}%
^{\prime\prime}(x)+2kf^{\prime}(x)\overline{g}^{\prime}(x)+k^{2}%
f(x)\overline{g}(x)\right)  dx
\end{align*}
is given by
\[
\mathcal{D}(A_{2})=V_{4}=\{f:[a,b]\rightarrow\mathbb{C}\mid f^{(j)}\in
AC[a,b]\text{ and }f^{(j)}(a)=-f^{(j)}(b)\text{ }(j=0,1,2,3);f^{(4)}\in
L^{2}[a,b]\}.
\]
Observe that $V_{2}=\mathcal{D}(A).$
\end{remark}


\begin{thebibliography}{99}                                                                                               %


\bibitem {AG}N. I. Akhiezer and I. M. Glazman, \emph{Theory of linear
operators in Hilbert space, }Dover Publications, New York, 1993.

\bibitem {BBW}B. M. Brown, C. Bennewitz, and R. Weikard, \emph{Spectral and
scattering theory for ordinary differential equations, Vol 1: Sturm-Liouville
Equations, }Springer Universitext, Springer Nature Switzerland AG, 2020.

\bibitem {LEW-Hermite}W. N. Everitt, L. L. Littlejohn, and R. Wellman,
\emph{The left-definite spectral theory for the classical Hermite differential
equation}\textit{,} J. Comput. Appl. Math.\textit{ }121(2000), 313-330.

\bibitem {LEW-Legendre}W. N. Everitt, L. L. Littlejohn, and R. Wellman,
\emph{Legendre polynomials, Legendre-Stirling numbers, and the left-definite
spectral analysis of the Legendre differential expression}$,$ J. Comput. Appl.
Math. 148(2002), 213-238.

\bibitem {ELKYW-Jacobi}W. N. Everitt, K. H. Kwon, L. L. Littlejohn, R.
Wellman, and G. J. Yoon, \emph{Jacobi-Stirling numbers, Jacobi polynomials,
and the left-definite analysis of the classical Jacobi differential
expression}, J. Comput. Appl. Math. 208 (2007), no. 1, 29--56.

\bibitem {LW}L. L. Littlejohn and R. Wellman, \emph{A general left-definite
theory for certain self-adjoint operators with applications to differential
equations, }J. Differential Equations, 181(2) (2002), 280-339.

\bibitem {LZ}L. L. Littlejohn and A. Zettl, \emph{Left-definite variations of
the classical Fourier expansion theorem, }Electronic Transactions on Numerical
Analysis 27(2005), 124-139.

\bibitem {Naimark}M. A. Naimark, \emph{Linear differential operators II,}
Frederick Ungar Publishing Co., New York, 1968.

\bibitem {Rudin}W. Rudin, \emph{Real and complex analysis, }3rd edition,
McGraw-Hill Series in Higher Education, McGraw-Hill, New York, 1987.

\bibitem {SS}F. W. Sch\"{a}fke and A. Schneider, \emph{SS-Hermitesche
Randeigenwertprobleme I}, Math. Ann., 162(1965), 9-26.

\bibitem {Weyl}H. Weyl, \emph{\"{U}ber gew\"{o}hnliche Differentialgleichungen
mit Singularit\"{a}ten und die zugeh\"{o}rigen Entwicklungen willk\"{u}rlicher
Funktionen, }Math. Annalen, 68(1910), 220-269.

\bibitem {Zettl}A. Zettl, \emph{Sturm-Liouville Theory, }Mathematical Surveys
and Monographs, Volume 121, American Mathematical Society, Providence, Rhode
Island, 2005.
\end{thebibliography}
\end{document}